\documentclass{amsart}
\usepackage{graphicx} 
\usepackage[utf8]{inputenc}
\usepackage{amsmath}
\usepackage{tikz}
\usetikzlibrary{shapes.geometric}
\usepackage{amsthm}
\usepackage{amsfonts}
\usepackage{caption}
\usepackage{amssymb}
\usepackage{verbatim}
\usepackage{subcaption}
\usepackage{changepage}
\newtheorem{theorem}{Theorem}
\newtheorem{lemma}[theorem]{Lemma}

\newtheorem{corollary}[theorem]{Corollary}
\newtheorem{definition}[theorem]{Definition}

\title{Ducci on $\mathbb{Z}_m^3$ and the Max Period}
\author{Mark L. Lewis}
\address{Department of Mathematical Sciences\\
Kent State University\\
Kent, OH 44242}
\email{lewis@math.kent.edu}
\author{Shannon M. Tefft}
\address{Department of Mathematical Sciences\\
Kent State University\\
Kent, OH 44242}
\email{stefft@kent.edu}
\date{January 2024}


\begin{document}
\begin{abstract}
    Let $D(x_1, x_2, ..., x_n)=(x_1+x_2 \;\text{mod} \; m, x_2+x_3 \; \text{mod} \; m, ..., x_n+x_1 \; \text{mod} \; m)$ where $D \in End(\mathbb{Z}_m^n)$ be the Ducci function. The sequence $\{D^k(\mathbf{u})\}_{k=0}^{\infty}$ will eventually enter a cycle. If $n=3$, we aim to establish the longest a cycle can be for a given $m$.
\end{abstract}
\maketitle

\section{Introduction}\label{intro}
\indent Let $D: \mathbb{Z}_m^n \to \mathbb{Z}_m^n$ be defined as 
\[D(x_1, x_2, ..., x_n)=(x_1+x_2 \; \text{mod} \; m, x_2+x_3 \; \text{mod} \; m, ..., x_n+x_1 \; \text{mod} \; m)\]

$D$ is commonly known as the \textbf{Ducci function}, with \cite{Breuer,Ehrlich,Glaser} being a few sources that refer to it in this way. The sequence $\{D^k(\mathbf{u})\}_{k=0}^{\infty}$ for a tuple $\mathbf{u} \in \mathbb{Z}_m^n$ is known as the \textbf{Ducci sequence of} $\mathbf{u}$. 

\indent To give an example of a Ducci sequence, consider $(3,4,4) \in \mathbb{Z}_6^3$ and the first eight terms in its Ducci sequence: $(3,4,4),(1,2,1),(3,3,2),(0,5,5),$ $(5,4,5),$ $(3,3,4),$ $(0,1,1),(1,2,1)$. Notice that the rest of the Ducci sequence for $(3,4,4)$ will cycle through the tuples $(1,2,1),(3,3,2),(0,5,5),(5,4,5),(3,3,4),$ $(0,1,1)$. These six tuples are called the Ducci cycle of the Ducci sequence of $(3,4,4)$. Generally speaking,

\begin{definition}
As in \cite{Breuer, Chamberland, Dular, Paper1}, the \textbf{Ducci cycle of} $\mathbf{u}$ is  $\{\mathbf{v} \in \mathbb{Z}_m^n \mid \exists k \in \mathbb{Z}^+, l \in \mathbb{Z}^+ \cup \{0\} \ni \mathbf{v}=D^{l+k}(\mathbf{u})=D^k(\mathbf{u})\}$.
The \textbf{Length of } $\mathbf{u}$, $\mathbf{Len(u)}$, is the smallest $l$ satisfying the equation $\mathbf{v}=D^{l+k}(\mathbf{u})=D^k(\mathbf{u})$ for some $v \in \mathbb{Z}_m^n$ \and the \textbf{Period of} $\mathbf{u}$, $\mathbf{Per(u)}$, is the smallest $k$ that satisfies the equation, as it is stated in Definition 3 in \cite{Breuer} and \cite{Ehrlich, Paper1}. 
\end{definition}

For ease, if $\mathbf{v} \in \mathbb{Z}_m^n$ is in the Ducci cycle for some $\mathbf{u} \in \mathbb{Z}_m^n$, we may say that $\mathbf{v}$ is in a Ducci cycle. In addition to this, because $\mathbb{Z}_m^n$ is finite, it must be true that every Ducci cycle in $\mathbb{Z}_m^n$ enters a cycle. 

\indent The Ducci sequence of $(0,0,...,0,1) \in \mathbb{Z}_m^n$ is a particularly important Ducci sequence and is called the \textbf{basic Ducci sequence} of $\mathbb{Z}_m^n$. We believe this sequence name was first coined by \cite{Ehrlich} on page 302 and is also used by \cite{Breuer,Glaser, Paper1}.
Denote $L_m(n)=Len(0,0,...,0,1)$ and $P_m(n)=Per(0,0,...,0,1)$ as it was defined on page 2 of \cite{Paper1} and is similar in notation to those used in \cite{Breuer, Dular, Ehrlich}. The basic Ducci sequence, $L_m(n)$, and $P_m(n)$ are particularly important because for $\mathbf{u} \in \mathbb{Z}_m^n$, $L_m(n) \geq Len(\mathbf{u})$ and $Per(\mathbf{u})|P_m(n)$, which is Lemma 1 of \cite{Breuer}. Note that because of this theorem, $P_m(n)$ provides a maximum value for what the period and length of a tuple will be for a given $n,m$. Moreover, it further limits the possible values of the period of a tuple for a given $n,m$ to the divisors of $P_m(n)$. 

\indent In this paper, we focus on when $n=3$ and the value of $P_m(3)$ for a given $m$.  The goal is to provide a comprehensive guide to finding the value of $P_m(3)$ by proving the following theorem:\\
\begin{theorem}\label{Period_for_n=3}
Let $n=3$
\begin{enumerate} 
    \item If $m=2^l$ for $l \geq 2$, then $P_m(3)=6$\\
    \item $P_2(2)=3$\\
    \item $P_3(3)=6$\\
    \item If $p$ is an odd prime and $k>0$ is the smallest number such that $2^k \equiv 1 \; \text{mod} \; p$, then $P_p(3)=\text{lcm}(6,k)$\\
    \item If $N \geq 1$ is the smallest integer such that $2^{p-1} \equiv 1 \; \text{mod} \; p^N$ and $2^{p-1} \not \equiv 1 \; \text{mod} \; p^{N+1}$, then $P_{p^{N+k}}(3)=p^kP_p(3)$ for $k \geq 1$.\\
    \item Let $m=2^lp_1^{k_1}p_2^{k_2}\cdots p_r^{k_r}$ where $p_i$ are distinct primes, $l \geq 0$, and $k_i \in \mathbb{Z}^+$ for $1 \leq i \leq r$. Let $N_i$ be the smallest integer such that $2^{p_i-1} \equiv 1 \; \text{mod} \; p_i^{N_i}$ and $2^{p_i-1} \not \equiv 1 \; \text{mod} \; p_i^{N_i+1}$ and let $\alpha_i=0$ if $k_i \leq N_i$ and $\alpha_i=k_i-N_i$ if $k_i>N_i$. Then
     \[P_m(3)=lcm\{P_{p_i^{k_i}}(3) \; |\; 1 \leq i \leq r\}\]\[=lcm\{P_{p_i}(3)\; |\; 1 \leq i \leq r\}\prod_{i=1}^r p_i^{\alpha_i}\] 
\end{enumerate}
\end{theorem}
 
\indent We note that $(1)$ and $(2)$ of this theorem were proved in Proposition 5.2 in  \cite{Dular} and that for $m$ odd and not divisible by a prime $p$ such that $2^{p-1} \equiv 1 \; \text{mod} \; p^2$, $(5)$ and $(6)$ follows from Theorem 5.11 in \cite{Dular}. 

\indent We would like to thank Professor Bruno Dular for bringing his paper \cite{Dular} to our attention and for his comments that helped us improve the paper. 

\indent The work in this paper was done while the second author was a Ph.D. student at Kent State University under the advisement of the first author and will appear as part of the second author's dissertation.
\section{Background}\label{background}
\indent There are a variety of ways that the Ducci function has been defined in the literature. Most papers will talk about $\bar{D}(x_1, x_2, ..., x_n)=(|x_1-x_2|, |x_2-x_3|, ..., |x_n-x_1|)$ as an endomorphism on $(\mathbb{Z}^+)^n$, with some examples being \cite{Ehrlich, Freedman, Glaser, Furno} or \cite{Breuer} defined it similarly on $\mathbb{Z}^n$. Some sources, like \cite{Brown, Chamberland,Schinzel}, define $\bar{D}$ as an endomorphism on $\mathbb{R}^n$ instead of on $(\mathbb{Z}^n)^+$. Of course, the results found for Ducci on $\mathbb{Z}^n$ are different than those on $\mathbb{R}^n$. 

\indent Focusing on the Ducci function defined on $\mathbb{Z}^n$ and $(\mathbb{Z}^+)^n$, it is well known that all Ducci sequences will eventually enter a cycle. \cite{Ehrlich,Glaser, Furno} have all talked about why this is the case and it was proved for Ducci on $\mathbb{R}^n$ by \cite{Schinzel} in Theorem 2. 
In addition to this, if $D^k(\mathbf{u})$ is in the Ducci cycle of $\mathbf{u} \in (\mathbb{Z}^+)^n$, then all of the entries in $D^k(\mathbf{u})$ belong to $\{0, c\}$ for some $c \in \mathbb{Z}^+$, as it was proved in Lemma 3 of \cite{Furno}. In addition to this, $D(\lambda \mathbf{u})=\lambda D(\mathbf{u})$ for $\lambda \in \mathbb{Z}$ and $\mathbf{u} \in \mathbb{Z}^n$, which makes the Ducci function defined on $\mathbb{Z}_2^n$ very important when it comes to studying the case defined on $\mathbb{Z}^n$, which is discussed in  \cite{Breuer, Ehrlich, Glaser, Furno}. Theorem 1 in \cite{Schinzel} proves a similar finding for Ducci on $\mathbb{R}^n$, which is that after a Ducci sequence reaches a limit point, all of the coordinates of the tuples of the sequence will either be $0$ or $c$ for some $c \in \mathbb{R}$.

\indent We now return to our case of Ducci on $\mathbb{Z}_m^n$. This case was first examined in \cite{Wong}, and was also explored in \cite{Breuer,Dular, Paper1}. First, $D \in End(\mathbb{Z}_m^n)$ which is how it was defined in Definition 1 of \cite{Breuer} and was also proved in \cite{Paper1}. Additionally from \cite{Paper1}, $D(\lambda \mathbf{u})=\lambda D(\mathbf{u})$ for this case as well. 

\indent Let $K(\mathbb{Z}_m^n)=\{\mathbf{u} \in \mathbb{Z}_m^n \mid \mathbf{u} \; \textit{is in the Ducci cycle for some tuple} \; \mathbf{v} \in \mathbb{Z}_m^n\}$ . This specific notation was first used in Definition 4 of \cite{Breuer} and then again in \cite{Paper1}. This set is significant because $K(\mathbb{Z}_m^n)$ is a subgroup of $\mathbb{Z}_m^n$ as stated in \cite{Breuer} and proved in Theorem 1 of \cite{Paper1}. 

\indent Let $H$ be the endomorphism of $\mathbb{Z}_m^n$ such that $H(x_1, x_2, ..., x_n)=(x_2, x_3, ..., x_n, x_1)$. If $I$ is the identity endomorphism on $\mathbb{Z}_m^n$, then $D=I+H$ and $D$ and $H$ commute, as it was first defined on page 302 in \cite{Ehrlich} and also used in \cite{Breuer, Dular, Glaser, Paper1}. Because of this, the Ducci cycle of $H(\mathbf{u})$ for $\mathbf{u} \in \mathbb{Z}_m^n$ is $\{H(D^k(\mathbf{u}))\}_{k=0}^{\infty}$ and if $\mathbf{u} \in K(\mathbb{Z}_m^n)$, then $H^s(\mathbf{u}) \in \mathbb{Z}_m^n$ for $0 \leq s <n$, which was seen in \cite{Paper1}.

\indent Suppose for a given tuple $\mathbf{u} \in \mathbb{Z}_m^n$, there exists $\mathbf{v} \in \mathbb{Z}_m^n$ such that $D(\mathbf{v})=\mathbf{u}$. Then we call $\mathbf{v}$ a \textbf{predecessor} of $\mathbf{u}$. The earliest instance of this term we found was on page 259 of \cite{Furno} and was also used in \cite{Breuer,Glaser, Paper1}.

\indent We now discuss primes, $p$, that satisfy the condition $2^{p-1} \equiv 1 \; \text{mod} \; p^2$. These are known as Wieferich primes, with \cite{Crandall} being one source that uses this name. As of the writing of this paper, only two Wieferich primes have been found less than $4.97*10^{17}$ as seen in \cite{OEIS}, which are $1093$ and $3511$. Note that if $o(m)$ is the order of $2$ mod $m$, then $o(p^3)=p*o(p)$ for $p \in \{1903, 3511\}$. Therefore, as of now, there are no primes yet found such that $2^{p-1} \equiv 1 \; \text{mod} \; p^j$ for $j \geq 3$.

\indent A question about Wieferich primes that is still unknown is whether or not there are an infinite number of Wieferich primes or not, with \cite{Crandall} asking on page 446 if there are even any larger than $3511$. For a number $m=2^lp_1^{k_1}p_2^{k_2} \cdots p_r^{k_r}$ where $p_i$ is a prime that is not Wieferich for $1 \leq i \leq r$, then assuming that Theorem \ref{Period_for_n=3} is true, $P_m(3)=lcm\{P_{p_i}(3) | 1 \leq i \leq r\}\prod_{i=1}^rp_i^{k_i-1}$. These questions about how many Wieferich primes there are leads to this question: How often will we need to make an exception for them when finding $P_m(3)$?

\indent Let us once more look at our example of the Ducci sequence of $(3,4,4) \in \mathbb{Z}_6^3$. To provide a better understanding of what Ducci sequences look like, we can create a transition graph that maps all of the sequences in $\mathbb{Z}_6^3$ and focus on the connected component that contains $(3,4,4)$, which we provide below.

\begin{figure}
\centering
\begin{adjustwidth}{-50 pt}{-50 pt}

\begin{tikzpicture}[node distance={30mm}, thick, main/.style = {draw, circle}]
\node[main](1){$(0,1,1)$};
\node[main](2)[above left of=1]{$(0,0,1)$};
\node[main](3)[right of=1]{$(1,2,1)$};
\node[main](4)[above right of=3]{$(3,4,4)$};
\node[main](5)[below right of=3]{$(3,3,2)$};
\node[main](6)[right of=5]{$(4,5,4)$};
\node[main](7)[below left of=5]{$(0,5,5)$};
\node[main](8)[below right of=7]{$(0,0,5)$};
\node[main](9)[left of=7]{$(5,4,5)$};
\node[main](10)[below left of=9]{$(3,2,2)$};
\node[main](11)[above left of=9]{$(3,3,4)$};
\node[main](12)[left of=11]{$(2,1,2)$};

\draw[->](2)--(1);
\draw[->](1)--(3);
\draw[->](4)--(3);
\draw[->](3)--(5);
\draw[->](6)--(5);
\draw[->](5)--(7);
\draw[->](8)--(7);
\draw[->](7)--(9);
\draw[->](10)--(9);
\draw[->](9)--(11);
\draw[->](12)--(11);
\draw[->](11)--(1);

\end{tikzpicture}
\end{adjustwidth}
\caption{Transition Graph for $\mathbb{Z}_6^3$}
\end{figure}
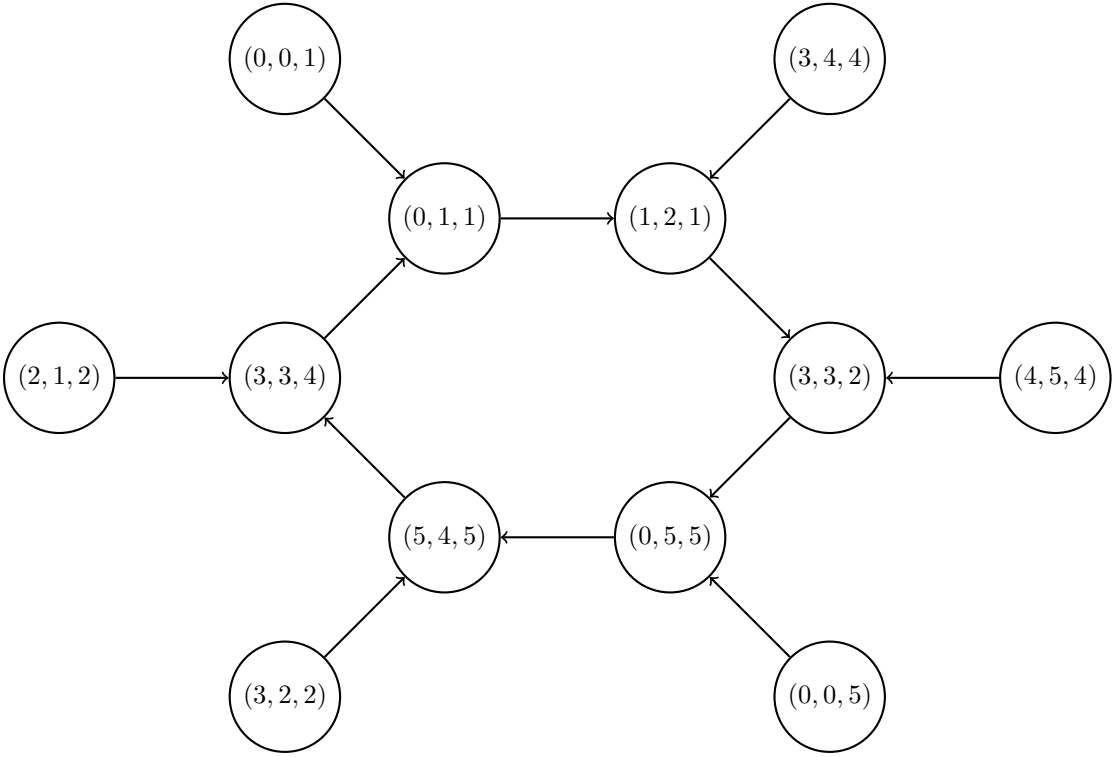

This component also contains $(0,0,1)$, which means this example provides us with the basic Ducci sequence for $\mathbb{Z}_6^3$. Taking advantage of the definitions given in Section \ref{intro}, $L_6(3)=1$, $Len(3,4,4)=1$, and $Len(1,2,1)=0$. Furthermore, $(0,0,1)$, $(3,4,4)$, and all of the tuples seen in this component have slightly different sequences, but they all have the same Ducci cycle. As a result, all of the tuples $\mathbf{u}$ in the component satisfy $Per(\mathbf{u})=6$. Since $(0,0,1)$ is among these tuples, $P_6(3)=6$.

\indent Note that every tuple in our graph that has a predecessor has exactly 2 predecessors. If $n$ is odd and $m$ is even, this is always the case:
\begin{theorem}\label{n odd m even pred}
For $n$ odd and $m$ even, every tuple that has a predecessor has exactly 2 predecessors. If one predecessor is $(x_1, x_2, ..., x_n)$, then the other predecessor is $(\frac{m}{2}+x_1,\frac{m}{2}+x_2,..., \frac{m}{2}+x_n)$.
\end{theorem}
\begin{proof}
Let $n$ be odd and $m$ be even. Notice
\[D(x_1, x_2, ..., x_n)=(x_1+x_2,x_2+x_3,...,x_n+x_1)\]
\[D(\frac{m}{2}+x_1,\frac{m}{2}+x_2,..., \frac{m}{2}+x_n)=(m+x_1+x_2,m+x_2+x_3,...,m+x_n+x_1)\]\[=(x_1+x_2,x_2+x_3,...,x_n+x_1)\]
So if $(x_1, x_2, ..., x_n)$ is a predecessor to a tuple, $(\frac{m}{2}+x_1,\frac{m}{2}+x_2,..., \frac{m}{2}+x_n)$ is a predecessor to that same tuple. Next we prove that if a tuple has a predecessor, it has exactly 2 predecessors.\\
\indent Suppose $\mathbf{u}=(z_1,z_2,...,z_n)$ has 2 predecessors $(x_1, x_2, ..., x_n)$ and $(y_1,y_2,...,y_n)$. Then we have $x_1+x_2 \equiv y_1+y_2 \; \text{mod} \; m,\; x_2+x_3 \equiv y_2+y_3 \; \text{mod} \; m \;, ...,  x_1+x_n \equiv y_1+y_n \; \text{mod} \; m$. Since the $x_i$ and $y_j$ are each at most $m-1$ for every $i,j, 1 \leq i,j \leq n$, we get 
\[x_1+x_2=y_1+y_2+ \alpha_1m\]
\[x_2+x_3=y_2+y_1+ \alpha_2m\]
\[\vdots\]
\[x_1+x_n=y_1+y_n+ \alpha_nm\]
where each of the $\alpha_i \in \{-1,0,1\}$ for $1 \leq i \leq n$. Subtracting the second equation from the first, we get $x_1-x_3=y_1-y_3+(\alpha_1-\alpha_2)m$. Adding this to the third equation and continuing this pattern, we get 
\[x_1+x_4=y_1+y_4(\alpha_1-\alpha_2+\alpha_3)m\]
\[\vdots\]
\[x_1-x_n=y_1-y_n+(\alpha_1-\alpha_2+ \cdots -\alpha_{n-1})m\]
Now if we add this to the equation $ x_1+x_n= y_1+y_n + \alpha_nm$, we get
\[2x_1=2y_1+(\alpha_1-\alpha_2+ \cdots +\alpha_n)m\]
We therefore have 2 possible cases:
\begin{itemize}
    \item Case 1: $\alpha_1-\alpha_2+ \cdots +\alpha_n$ is even\\
    Here $x_1=y_1+\beta m$ where $\beta \in \mathbb{Z}$. Since $y_1-m<x_1<y_1+m$, the only possibility then is that $x_1=y_1$.\\
    \item Case 2: $\alpha_1-\alpha_2+ \cdots +\alpha_n$ is odd\\
    Here, $x_1=y_1+\gamma m+\frac{m}{2}$ where $\gamma \in \mathbb{Z}$. The only possibility then is $x_1=y_1+\frac{m}{2}$, which is the previously discussed case of this theorem.\\
 
\end{itemize}
\indent If $x_1=y_1$, then $x_2=y_2, x_3=y_3, ..., x_n=y_n$. If $x_1=y_1+\frac{m}{2}$, then $x_2=y_2+\frac{m}{2}, x_3+ \frac{m}{2}=y_3,..., x_n=y_n+\frac{m}{2}$.\\
\indent The Theorem follows from here.
\end{proof}
\indent For our next example, consider the connected component of the transition graph that contains the basic Ducci sequence for $\mathbb{Z}_3^3$.
\begin{figure}
\centering
\begin{tikzpicture}[node distance={30mm}, thick, main/.style = {draw, circle}]
\node[main](1){$(0,0,1)$};
\node[main](2)[right of=1]{$(0,1,1)$};
\node[main](3)[below right of=2]{$(1,2,1)$};
\node[main](4)[below left of=3]{$(0,0,2)$};
\node[main](5)[left of=4]{$(0,2,2)$};
\node[main](6)[above left of=5]{$(2,1,2)$};

\draw[->](1)--(2);
\draw[->](2)--(3);
\draw[->](3)--(4);
\draw[->](4)--(5);
\draw[->](5)--(6);
\draw[->](6)--(1);
    
\end{tikzpicture}
\caption{Transition Graph for $\mathbb{Z}_3^3$}
\end{figure}
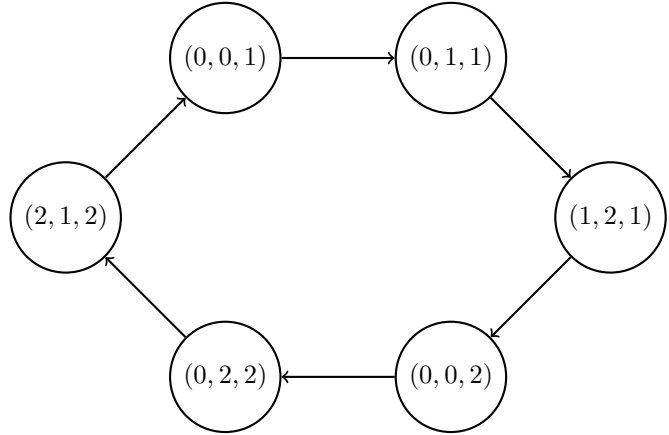
Because this is the basic Ducci sequence, it is evident that $P_3(3)=6$.
Note that every tuple, $\mathbf{u}$, in this component satisfies $Len(\mathbf{u})=0$ and $L_3(3)=0$. Since $(0,0,...,0,1)$ is in this component, we have $L_3(3)=0$, which implies $Len(\mathbf{u})=0$, $\mathbf{u} \in K(\mathbb{Z}_3^3)$ for every $\mathbf{u} \in \mathbb{Z}_3^3$, and $K(\mathbb{Z}_3^3)=\mathbb{Z}_3^3$. If $n,m$ are odd, then this is always true:
\begin{theorem}\label{n_odd_Length_0}
For $n$ odd and $m$ odd, $L_m(n)=0$.
\end{theorem}
\begin{proof}
For a given $n,m$, if $\mathbf{u}, \mathbf{v} \in \mathbb{Z}_m^n$ and $\mathbf{u}$ is a predecessor to $\mathbf{v}$, then either $\mathbf{u} \not \in K(\mathbb{Z}_m^n)$ and $Len(\mathbf{u})=Len(\mathbf{v})+1$, or $\mathbf{u} \in K(\mathbb{Z}_m^n)$ and $Len(\mathbf{u})=Len(\mathbf{v})=0$.

\indent Now let $n,m$ be odd and $\mathbf{u}=(\frac{m+1}{2},\frac{m-1}{2},\frac{m+1}{2},..., \frac{m-1}{2}, \frac{m+1}{2})$
Then $D(\mathbf{u})= (0,0,...,0,1)$. 
If $(0,0,..., 0,1) \not \in K(\mathbb{Z}_m^n)$, then $\mathbf{u} \not \in K(\mathbb{Z}_m^n)$ and $Len(\mathbf{u})>L_m(n)$. But this contradicts $L_m(n) \geq Len(\mathbf{u})$ for every $\mathbf{u} \in \mathbb{Z}_m^n$. Therefore, $(0,0,...,0,1) \in K(\mathbb{Z}_m^n)$ and $L_m(n)=0$.
\end{proof}
This theorem yields Proposition 6.1 in \cite{Dular} as a corollary, which says if $m,n$ are odd, then $K(\mathbb{Z}_m^n)=\mathbb{Z}_m^n$.
In Proposition 3.1 of \cite{Dular} it is proven that if $d|m$, then $P_d(n)|P_m(n)$. We wish to provide an alternative proof of this proposition with the next lemma, by additionally proving that there is a tuple in $\mathbb{Z}_m^n$ with period $P_d(m)$:
\begin{lemma}\label{divisor_divides_period}
Let $m,n$ be integers and suppose $m=dk$ for some $d,k \in \mathbb{Z}$. Then $Per(0,0,...,0,k)=P_d(m)$ where $(0,0,...,0,k) \in \mathbb{Z}_m^n$. Moreover, $P_d(n)|P_m(n)$.
\end{lemma}
\begin{proof}

\indent Let $m,n$ be integers and assume $m=dk$ for some $d,k \in \mathbb{Z}$. Let $l=L_d(n)$ and $x=P_d(n)$. 

\indent Now consider $(0,0,...,0,k) \in \mathbb{Z}_m^n$. Then 
\[D^{l+x}(0,0,...,k)=(ka_{l+x,n}, ka_{l+x,n-1}, ..., ka_{l+x,1})\]
\[=(ka_{l,n}, ka_{l,n-1}, ..., ka_{l,1})\]
Because $x=P_d(n)$ gives us that for every $s$, $1 \leq s \leq n$, we have 
\[a_{l+x,s} \equiv a_{l,s} \; \text{mod} \; d \]
\[ka_{l+x,s} \equiv ka_{l,s} \; \text{mod} \; m\]
Since $x$ is the smallest value of $y$ that we have $a_{l+y,s} \equiv a_{l,s} \; \text{mod} \; d$, we have $Per(0,0,...,0,k)=P_d(n)$. Therefore $Per(0,0,...,0,k)|P_m(n)$ gives us that $P_d(n)|P_m(n)$.

\end{proof}

\section{The coefficients $a_{r,s}$}\label{Findings_General}

     Consider $\mathbf{u}=(x_1, x_2, ..., x_n) \in \mathbb{Z}_m^n$. If we look at the first few tuples in the Ducci sequence of $\mathbf{u}$, we would have
     \[(x_1, x_2, ..., x_n)\]
     \[(x_1+x_2, x_2+x_3, ..., x_n+x_1)\]
     \[(x_1+2x_2+x_3, x_2+2x_3+x_4, ..., x_n+2x_1+x_2\]
     \[(x_1+3x_2+3x_3+x_4, x_2+3x_3+3x_4+x_5, ..., x_n+3x_1+3x_2+x_3)\]
     \[\vdots\]
     Note that the coefficient on $x_1$ in the first coordinate is the same as the coefficient on $x_2$ in the second coordinate of $D^r(\mathbf{u})$ for $0 \leq r \leq 3$. Similarly, the coefficient on $x_2$ in the first coordinate is the same as that on $x_3$ in the second coordinate. This pattern will continue, so we define $a_{r,s}$ to be the coefficient on $x_{s-i+1}$ in the $i$th coordinate of $D^r(\mathbf{u})$ where $r \geq 0$ and $1 \leq s \leq n$, which was first used on page 6 of \cite{Paper1}. So 
     \[D^r(x_1, x_2, ..., x_n)=(a_{r,1}x_1+a_{r,2}x_2+ \cdots +a_{r,n}, a_{r,n}x_1+a_{r,1}x_2+\cdots +a_{r,n-1}x_n,...,a_{r,2}x_1+a_{r,3}x_2+\cdots+a_{r,1}x_n)\]
     Also note that this means $D^r(0,0,...,0,1)=(a_{r,n},a_{r,n-1},...,a_{r,1})$. 

    \begin{theorem} \label{coefficients_sum_to_power_of_2}
    Let $r \geq 0$. Then
    \[\sum_{s=1}^n a_{r,s}=2^r\]
\end{theorem}
\begin{proof}
    We prove this by induction with the basis case being $r=0$.\\
    \textbf{Inductive:} Assume that $\sum_{i=1}^n a_{r-1,s}=2^{r-1}$.  Using Theorem 5 of \cite{Paper1}, $a_{r,s}=a_{r-1,s}+a_{r-1,s-1}$ so
    \[\sum_{i=1}^n a_{r,s} =\sum_{i=1}^n a_{r-1,s}+a_{r-1, s-1}\]
    \[=\sum_{i=1}^n a_{r-1,s}+\sum_{i=1}^n a_{r-1,s-1}=\sum_{i=1}^n a_{r-1,s}+\sum_{i=1}^n a_{r-1, s}\]
    \[=2^{r-1}+2^{r-1}=2^r\]
    and the theorem follows. 
\end{proof}

\indent For the rest of the paper, we will work specifically in the case where $n=3$. For ease, we use $a_r=a_{r,1}, b_r=a_{r,2},$ and $c_r=a_{r,3}$ and define some lemmas for $a_r, b_r, c_r$ that are more specific to the $n=3$ case. Theorem 5 of \cite{Paper1} tells us that for all $n$, $a_{r+t,s}=\sum_{i=1}^n a_{t,i}a_{r,s-i+1}$ where $t \geq 0$. This gives us the following corollary for $n=3$. 

\begin{corollary} \label{coefficient_sum_n=3}
Let $n=3$. For $r> t >1$,
\[a_{r+t}=a_ta_{r}+b_tc_{r}+c_tb_{r}\]
\[b_{r+t}=a_tb_{r}+b_ta_{r}+c_tc_{r}\]
\[c_{r+t}=a_tc_{r}+b_tb_{r}+c_ta_{r}\]
\end{corollary}


\indent When $n=3$, $a_r, b_r,$ and $c_r$, are all very close in value, which is typically not the case when $n>3$. We can define the next lemma to be able to give $a_r, b_r$ and $c_r$ in terms of each other if we know the others:

 \begin{lemma}\label{coefficients_mod_6}
 

Let $n=3$. Then
\begin{itemize}
    \item If $r \equiv 0 \; \text{mod} \; 6$, $a_r=b_r+1=c_r+1$.\\
    \item If $r \equiv 1 \; \text{mod} \; 6$, then $c_r=a_r-1=b_r-1$.\\
    \item If $r \equiv 2 \; \text{mod} \; 6$, then $b_r=a_r+1=c_r+1$.\\
    \item If $r \equiv 3 \; \text{mod} \; 6$, then $a_r=b_r-1=c_r-1$.\\
    \item If $r \equiv 4 \; \text{mod} \; 6$, then $c_r=a_r+1=b_r+1$.\\
    \item If $r \equiv 5 \; \text{mod} \; 6$, then $b_r=a_r-1=c_r-1$.
    
\end{itemize}

\end{lemma}   
\begin{proof}
    We prove this via induction where the basis case is $0 \leq r \leq 5$:
    \begin{center}
        \begin{tabular}{|c| c| c| c|} 
        \hline
        $r$ & $a_r$ & $b_r$ & $c_r$\\
         \hline  \hline
        0 & 1 & 0 & 0\\
         \hline
        1 & 1 & 1 & 0\\
         \hline
        2 & 1 & 2 & 1\\
         \hline
        3 & 2 & 3 & 3\\
         \hline
        4 & 5 & 5 &6\\
         \hline
        5 & 11 & 10 &11\\
         \hline
        \end{tabular}
    \end{center}
    
Assume that the Lemma is true for $r'<r$ and let $l \in \mathbb{Z}^+$.
For $r \equiv 0 \; \text{mod} \; 6$, write $r=6l$:
\[a_{6l}=a_{6l-1}+c_{6l-1}=b_{6l-1}+1+c_{6l-1}=b_{6l}+1\]
For $r \equiv 1 \; \text{mod} \; 6$, write $r=1+6l$:
\[c_{1+6l}=c_{6l}+b_{6l}=c_{6l}+a_{6l}-1=a_{6l}-1\]
For $r \equiv 2 \; \text{mod} \; 6$, write $r=2+6l$:
\[b_{6l+2}=a_{1+6l}+b_{1+6l}=a_{1+6l}+c_{1+6l}+1=a_{2+6l}+1\]
For $r \equiv 3 \; \text{mod} \; 6$, write $r=3+6l$:
\[a_{3+6l}=a_{2+6l}+c_{2+6l}=a_{2+6l}+b_{2+6l}-1=b_{3+6l}-1\]
For $r \equiv 4 \; \text{mod} \; 6$, write $r=4+6l$:
\[c_{4+6l}=c_{3+6l}+b_{3+6l}=c_{3+6l}+a_{3+6l}+1=a_{4+6l}+1\]
For $r \equiv 5 \; \text{mod} \; 6$, write $r=5+6l$:
\[b_{6l+5}=a_{4+6l}+b_{4+6l}=a_{4+6l}+c_{4+6l}-1=a_{4+6l}-1\]

\end{proof}

\indent We can go further into the relationships that $a_r, b_r, c_r$ have with each other for a given $r$, namely that $a_{r+1}, b_{r+1}, c_{r+1}$ are roughly double the $a_r, b_r, c_r$. We explore this in the next lemma.

\begin{lemma}\label{evenseqcoefficient}
 For $l \geq 2$, 
    \[l \equiv 0 \; \text{mod} \; 3, \; a_l=2a_{l-1}=2c_{l-1}\]
    \[l \equiv 1 \; \text{mod} \; 3, \; c_l=2c_{l-1}=2b_{l-1}\]
    \[l \equiv 2 \; \text{mod} \; 3, \; b_l=2b_{l-1}=2a_{l-1}\]
\end{lemma}
\begin{proof}
If the first equalities are true, then the second equalities will follow from Lemma \ref{coefficients_mod_6}.
    Lemma \ref{coefficients_mod_6} further gives us 
    \begin{itemize}
        \item $l \equiv 0 \; \text{mod} \; 3:$
        \[a_l=a_{l-1}+c_{l-1}=a_{l-1}+a_{l-1}=2a_{l-1}\]
         \item $l \equiv 1 \; \text{mod} \; 3:$
        \[c_l=c_{l-1}+b_{l-1}=c_{l-1}+c_{l-1}=c_{l-1}\]
         \item $l \equiv 2 \; \text{mod} \; 3:$
        \[b_l=b_{l-1}+a_{l-1}=b_{l-1}+b_{l-1}=2b_{l-1}\]
    \end{itemize}
\end{proof}


\indent We can now use the information we have about the $a_r, b_r, c_r$ to gain a result that will tell us about the value of $P_m(3)$.

\begin{theorem}\label{6_divides_Period}
For $m \geq 3$, $6|P_m(3)$
\end{theorem}
\begin{proof}
Take $l \in \mathbb{Z}^+ \cup \{0\}$ large enough so that $D^l(0,0,1)=(c_l, b_l, a_l) \in K(\mathbb{Z}_m^3)$ and $l \equiv 0 \; \text{mod} \; 6$. Let $x=P_m(3)$. Then we should have that 
\[a_{l+x} \equiv a_l \; \text{mod} \; m\]
\[b_{l+x} \equiv b_l \; \text{mod} \; m\]
\[c_{l+x} \equiv c_l \; \text{mod} \; m\]
Notice that $l \equiv 0 \; \text{mod} \; 6$ gives us that $b_l=c_l=a_l-1$. If $x \not\equiv 0 \; \text{or} \; 3 \; \text{mod} \; 6$, then 
\[c_{l+x} \neq b_{l+x} \equiv b_l \; \text{mod} \; m \equiv c_l \; \text{mod} \; m\]
Which implies $c_{l+x} \not \equiv c_l \; \text{mod} \; m$ which contradicts $x=P_m(3)$.\\
\indent If we have that $x \equiv 3 \; \text{mod} \; 6$, and $m>2$ then 
\[a_{l+x}+1=b_{l+x} \equiv b_l \; \text{mod} \; m \equiv a_l-1 \; \text{mod} \; m\]
which implies $a_{l+x} \not \equiv a_l \; \text{mod} \; m$ which again contradicts $x=P_m(3)$. Therefore $6|P_m(3)$ for $m>2$.
\end{proof}

\section{Value of $P_m(3)$}\label{Period_Section}
\indent There are still a few more lemmas and theorems that we will need in order to prove Theorem \ref{Period_for_n=3}.
\begin{theorem}\label{Overall_Length}
Let $n=3$. If $m=2^lm_1$ where $m_1$ is odd, then $L_m(3)=l$.
\end{theorem}
\begin{proof}
We prove this via induction on $l$.\\
\textbf{Basis} $\mathbf{l=0}$: In this case, $m$ is odd and $L_m(3)=0$ by Theorem \ref{n_odd_Length_0}.\\
\textbf{Inductive:} Assume that for $l'<l$ and $m'=2^{l'}m_1$, $L_{m'}(3)=l'$ where $m_1$ is odd. 
Let $x=P_{2^{l-1}m_1}(3)$. By induction, $L_{2^{l-1}m_1}(3)=l-1$. This means that in $\mathbb{Z}_{2^{l-1}m_1}^3$, $D^{l-1+x}(0,0,1)=D^{l-1}(0,0,1)$, or that for $s, 1 \leq s \leq 3$
\[a_{l-1+x,s} \equiv a_{l-1,s} \; \text{mod} \; 2^{l-1}m_1\]
By Lemma \ref{divisor_divides_period} and Theorem \ref{6_divides_Period} $P_{2^{l-1}m_1}(3)|P_m(3)=6j$ for some $j \in \mathbb{Z}^+$. By induction, $L_{2^{l-1}m_1}(3)=l-1$ implies $D^{l-2}(0,0,1)=(c_{l-2},b_{l-2},a_{l-2}) \not \in K(\mathbb{Z}_{2^{l-1}m_1}^3)$ but $D^{l-1}(0,0,1)=(c_{l-1},b_{l-1},a_{l-1}) \in K(\mathbb{Z}_{2^{l-1}m_1}^3)$. For $1 \leq s \leq 3$, we get
\[a_{l-2+6j,s} \not \equiv a_{l-2,s} \; \text{mod} \; 2^{l-1}m_1\]
\[a_{l-1+6j,s} \equiv a_{l-1,s} \; \text{mod} \; 2^{l-1}m_1\]
Suppose $l \equiv 0 \; \text{mod} \; 3$. By Lemma \ref{evenseqcoefficient}, 
$a_{l}=2a_{l-1}$ and $a_{l+6j}=2a_{l+6j-1}$. Then
\[a_{l-1+6j} \equiv a_{l-1} \; \text{mod} \; 2^{l-1}m_1\]
\[2a_{l-1+6j} \equiv 2a_{l-1} \; \text{mod} \; 2^lm_1\]
\[a_{l+6j} \equiv a_l \; \text{mod} \; 2^lm_1\]
Here, $l \equiv 0 \; \text{mod} \; 3$, which gives $b_l=c_l=a_l \pm 1$ and 
\[b_{l+6j}=a_{l+6j} \pm 1\equiv a_l \pm 1\; \text{mod} \; 2^lm_1 \equiv b_l \; \text{mod} \; m\]
and similarly $c_{l+6j} \equiv c_l \; \text{mod} \; m$. We also have $b_{l-1}=2b_{l-2}$ which gives us
\[b_{l-2+6j} \not \equiv b_{l-2} \; \text{mod} \; 2^{l-1}m_1\]
\[2b_{l-2+6j} \not \equiv 2b_{l-2} \; \text{mod} \; 2^lm_1\]
\[b_{l-1+6j} \not \equiv b_{l-1} \; \text{mod} \; 2^lm_1\]
The proofs for $l \equiv 1,2 \; \text{mod} \; 3$ are similar. It then follows then that $L_m(3)=l$.
\end{proof}

\indent Although \cite{Dular} proved $(2)$ of Theorem \ref{Period_for_n=3}, we wish to provide an alternative proof. To give this proof, we will need the following lemma:

\begin{lemma}\label{HclosedTheorem2}
    Let $n=3$ and $m=2^l$ for $l>0$. If $\mathbf{u} \in K(\mathbb{Z}_m^3)$, then $D^2(\mathbf{u})=H(\mathbf{u})$.
\end{lemma}
\begin{proof}
    Let $m=2^l$. We know $L_m(3)=l$ by Theorem \ref{Overall_Length}. It suffices to show $D^{l+2}(\mathbf{u})=H(D^l(\mathbf{u}))$ for every $\mathbf{u} \in \mathbb{Z}_m^3$. We can prove this by showing 
    \[a_{l+2} \equiv c_l \; \text{mod} \; 2^l\]
    \[b_{l+2} \equiv a_l \; \text{mod} \; 2^l\]
    \[c_{l+2} \equiv b_l \; \text{mod} \; 2^l\]
    Notice 
    \[a_{l+2}=a_la_2+b_lc_2+c_lb_2\]
    \[=a_l+b_l+2c_l\]
    \[=(a_l+b_l+c_l)+c_l\]
    \[=2^l+c_l\]
    \[\equiv c_l \; \text{mod} \; m\]
    Similarly, 
    \[b_{l+2}=a_lb_2+b_la_2+c_lc_2\]
    \[=2a_l+b_l+c_l\]
    \[=(a_l+b_l+c_l)+a_l\]
    \[=2^l+a_l\]
    \[\equiv a_l \; \text{mod} \; m\]
    and 
    \[c_{l+2}=a_lc_2+b_lb_2+c_la_2\]
    \[=a_l+2b_l+c_l\]
    \[=(a_l+b_l+c_l)+b_l\]
    \[=2^l+b_l\]
    \[\equiv b_l \; \text{mod} \; m\]
    The lemma follows from here.
\end{proof}

\indent We need one last lemma to prove Theorem \ref{Period_for_n=3}. This lemma will allow us to break down $b_r$ for certain values of $r \in \mathbb{Z}^+$.

\begin{lemma}\label{break_up_bs}
Assume $k \equiv 0 \; \text{mod} \; 6$. Then 
\[b_{jk}=\sum_{i=0}^{j-1} 3^{j-i-1}\binom{j}{j-i}b_k^{j-i}\]
\end{lemma}
\begin{proof}
First assume that $r_1, r_2 \equiv 0 \; \text{mod} \; 6$ and $r_1+r_2=r$. Then 
\[b_r=a_{r_1}b_{r_2}+b_{r_1}a_{r_2}+c_{r_1}c_{r_2}\]
\[=(b_{r_1}+1)b_{r_2}+b_{r_1}(b_{r_2}+1)+b_{r_1}b_{r_2}\]
\[=3b_{r_1}b_{r_2}+b_{r_1}+b_{r_2}\]
Then for our basis case $j=2$, we see it is true for 
\[b_{2k}=3b_k^2+2b_k\]
by letting $r_1=r_2=k$.\\
\textbf{Inductive Step}: Assume that 
\[b_{(j-1)k}=3^{j-2}b_k^{j-1}+3^{j-3}\binom{j-1}{j-2}b_k^{j-2}+ \cdots +(j-1)b_k\]
Then 
\[b_{jk}=3b_{(j-1)k}b_k+b_{(j-1)k}+b_k\]
\[=3(3^{j-2}b_k^{j-1}+3^{j-3}\binom{j-1}{j-2}b_k^{j-2}+ \cdots +(j-1)b_k)b_k+3^{j-2}b_k^{j-1}+3^{j-3}\binom{j-1}{j-2}b_k^{j-2}+ \cdots +(j-1)b_k+b_k\]
\[=3^{j-1}b_k^j+3^{j-2}(\binom{j-1}{j-2}+\binom{j-1}{j-1})b_k^{j-1}+3^{j-3}(\binom{j-1}{j-3}+\binom{j-1}{j-2})b_k^{j-2}+ \cdots+ jb_k\]
\[=3^{j-1}b_k^j+3^{j-2}\binom{j}{j-1}b_k^{j-1}+ \cdots+jb_k\]
 and the lemma follows.
\end{proof}

We now have the pieces we need to prove Theorem \ref{Period_for_n=3}. As mentioned in \ref{intro}, (1) and (2) were proved in Proposition 5.2 in \cite{Dular}. We provide an alternative proof of (1) because our proof provides a different perspective from \cite{Dular} of what the cycle of the basic Ducci sequence in $\mathbb{Z}_{2^l}^3$ looks like.

\begin{proof}[Proof of Theorem \ref{Period_for_n=3}]

 $\mathbf{(1)}$:  
 Let $m=2^l$ for $l \geq 2$. By Theorem \ref{6_divides_Period}, we know $6|P_{2^l}(3)$ and by Theorem \ref{HclosedTheorem2}, if $\mathbf{u} =D^k(0,0,1)$ for some $k \geq 1$ and $\mathbf{u} \in K(\mathbb{Z}_{2^l}^3)$, then $D^6(\mathbf{u})=H^3(\mathbf{u})=\mathbf{u}$ meaning that $P_{2^l}(3)|6$ and $P_{2^l}(3)=6$.
 
 $\mathbf{(2)}$: 
 We can prove $P_2(3)$ by showing there are $3$ tuples in the Ducci cycle of the basic Ducci sequence. This is how \cite{Dular} proves it, but he does not include the actual sequence in his proof. We provide the sequence here:
 \[(0,0,1)\]
 \[(0,1,1)\]
 \[(1,0,1)\]
 \[(1,1,0)\]
 \[(0,1,1)\]

 $\mathbf{(3)}$: When looking at the example of the basic Ducci sequence of $\mathbb{Z}_3^3$ in Figure 2 of Section \ref{background}, it was shown that $P_3(3)=6$.\\
 
 $\mathbf{(4)}$: Let $p>3$ be an odd prime. We first want to show that if $o(p)$ is the order of $2$ mod $p$, then $o(p)|P_p(3)$. This follows from Proposition 4.3 of \cite{Dular}, but we wish to provide an alternate proof that will also show that there is a tuple in $\mathbb{Z}_p^3$ that has period $o(p)$.
 
\indent Consider the Ducci sequence of a tuple $(x,x,x)$ for $x \in \mathbb{Z}_p$:
 \[(x,x,x)\]
 \[(2x,2x,2x)\]
 \[(4x,4x,4x)\]
 \[\vdots\]
 and we see that $D^{k_1}(x,x,x)=2^{k_1}(x,x,x)$ for $k_1\geq 0$. Therefore, the period of $(x,x,x)$ will be the smallest $k \geq 1$ such that $2^k \equiv 1 \; \text{mod} \; p$ and since $Per(x,x,x)|P_p(3)$, $k|P_p(3)$.
 
 \indent By Theorem \ref{6_divides_Period}, $6|P_p(3)$. So $\text{lcm}(6,k)|P_p(3)$.\\
 \indent Note that $\text{lcm}(6,k)=ik$ where $i=1,2,3,$ or $6$. Since $ik \equiv 0 \; \text{mod} \; 6$, we have that 
\[a_{ik}+b_{ik}+c_{ik}=2^{ik}=(2^k)^i\]
\[3b_{ik}+1 \equiv 1 \; \text{mod} \; p\]
\[3b_{ik} \equiv 0 \; \text{mod} \; p\]
\[b_{ik} \equiv 0 \; \text{mod} \; p\]
which would then give us
\[a_{ik} \equiv 1 \; \text{mod} \; p\]
\[c_{ik} \equiv 0 \; \text{mod} \; p\]
and therefore $D^{ik}(0,0,1)=(0,0,1)$ which means $P_p(3)|\text{lcm}(6,k)$ and $P_p(3)=\text{lcm}(6,k)$. \\

$\mathbf{(5)}$:  
A reminder that for non-Wieferich primes, $(5)$ follows from Theorem 5.11 in \cite{Dular}. We present a proof that work for all odd primes.

\indent Let $p$ be an odd prime. Assume that $x=P_p(3)$ and choose $N \geq 1$ to be the smallest integer such that $2^{p-1} \equiv 1 \; \text{mod} \; p^N$ and $2^{p-1} \not \equiv 1 \; \text{mod} \; p^{N+1}$. By Proposition 4.6 in \cite{Dular}, $2^{p-1} \not \equiv 1 \; \text{mod} \; p^{N+k}$ for $k \geq 1$.  

\indent We prove $(5)$ by induction with our basis case being to first prove that $P_{p^{N+1}}(3)=px$. 

By Lemma \ref{divisor_divides_period} we have $x|P_{p^{N+1}}(3)$. 

\indent Now define $o(m)$ to be the order of $2 \; \text{mod} \; m$ for $m$ odd. Note that by Proposition 4.6 in \cite{Dular}, $o(p^j)=o(p)$ for $1 \leq j \leq N$.  Since $x=lcm(6,o(p))$ by $(4)$, write $x=ko(p)$ for $k \in \{1,2,3,6\}$. By Theorem \ref{coefficients_sum_to_power_of_2},
\[a_x+b_x+c_x=2^{ko(p)}\]
\[3b_x+1=2^{ko(p)}\]
\[3b_x=2^{ko(p)}-1\]
\[3b_x=(2^{o(p)}-1)\sum_{j=0}^{k-1}2^{jo(p)}\]
The only case where we could have $(2^{o(p)}-1)|3$ is if $p=3$, but we know $3|b_x$, for that case. Therefore, we find that $(2^{o(p)}-1)|b_x$. If $N>1$, then we have $b_x \equiv 0 \; \text{mod} \; p^N$ and $P_{p^j}(3)=x$ for $1 \leq j \leq N$. We now wish to show that $b_x \not \equiv 0 \; \text{mod} \; p^{N+1}$, which we can prove by showing that $\sum_{j=0}^{k-1}2^{jo(p)} \not \equiv 0 \; \text{mod} \; p^{N+1}$. We break this up into cases based on the value of $k$. Assume $2^{o(p)} \equiv ip^N \; \text{mod} \; p^{N+1}$ for $1 \leq i <p$. Then

\textbf{Case 1} $\mathbf{k=1}$: $\sum_{j=0}^{k-1}2^{jo(p)}=1$

\textbf{Case 2} $\mathbf{k=2}$: $\sum_{j=0}^{1}2^{jo(p)}=2^{o(p)}+1 \not \equiv 0 \; \text{mod} \; p^{N+1}$

\textbf{Case 3} $\mathbf{k=3}$: 
\[\sum_{j=0}^{2}2^{jo(p)} \equiv 3ip^N+3 \; \text{mod} \; p^{N+1}\]
\[\equiv 3(ip^N+1) \; \text{mod} \; p^{N+1}\]
\[\not \equiv 0 \; \text{mod} \; p^{N+1}\]

\textbf{Case 4} $\mathbf{k=6}$: 
\[\sum_{j=0}^{5}2^{jo(p)} \equiv 15ip^N+6 \; \text{od} \; p^{N+1}\]
\[\equiv 3(5ip^N+2) \; \text{mod} \; p^{N+1}\]
\[\not \equiv 0 \; \text{mod} \; p^{N+1}\]

Therefore, $b_x \not \equiv 0 \; \text{mod} \; p^{N+1}$ and $P_{p^{N+1}}(3) \neq x$. 

\indent We now have $P_{p^j}(3)=x$ for $1 \leq j \leq N$. Then by Lemma \ref{break_up_bs}, 
\[b_{px}=3^{p-1}b_x^p+3^{p-2}\binom{p}{p-1}b_x^{p-1}+\cdots+pb_x\]
so $P_{p^{N+1}}(3)|px$. Since we know $P_{p^{N+1}}(3) \neq x$, $P_{p^{N+1}}(3)=px$. 

\textbf{Inductive:} Assume $P_{p^k}(3)=p^lP_p(3)$ where $1 \leq N \leq k-1$ and $l=k-N$. We want to show that $P_{p^{k+1}}(3)=p^{l+1}P_p(3)$. Let $x=P_p(3)$. First note 
\[b_{p^{l+1}x}=3^{p-1}b_{p^{l}x}^p+3^{p-2}\binom{p}{p-1}b_{p^{l}x}^{p-1}+\cdots+pb_{p^{l}x}\]
\[\equiv 0 \; \text{mod}\; p^{k+1}\]
because $b_{p^lx} \equiv 0 \; \text{mod} \; p^k$. Therefore, $P_{p^{k+1}}(3)|p^{l+1}x$. 

By Lemma \ref{divisor_divides_period}, $p^lx|P_{p^{k+1}}(3)$. So $P_{p^{k+1}}(3) \in \{p^lx, p^{l+1}x\}$. But 
\[b_{p^lx}=3^{p-1}b_{p^{l-1}x}^p+3^{p-2}\binom{p}{p-2}b_{p^{l-1}x}^{p-1}+ \cdots +pb_{p^{l-1}x}\]
\[\equiv p^k \; \text{mod} \; p^{k+1}\]
Because by induction, $b_{p^{l-1}x} \equiv p^{k-1} \; \text{mod} \; p^k$. Therefore, $P_{p^{k+1}}(3) = p^{l+1}x$.
 
\textbf{(6):} 
Let $m=2^lp_1^{k_1}p_2^{k_2} \cdots p_r^{k_r}$ where $p_i$ is an odd prime, $k_i \in \mathbb{Z}^+$, and $N_i$ is the smallest integer such that $2^{p_i-1} \equiv 1 \; \text{mod} \; p_i^{N_i}$ and $2^{p_i-1} \not \equiv 1 \; \text{mod} \; p_i^{N_i+1}$. Let
 $P=\text{lcm}\{x_i \; | \; x_i=P_{p_i}(3)\}\prod_{i=1}^rp_i^{\alpha_i}$ where $\alpha_i=0$ if $k_i \leq N_i$ and $\alpha_i=k_i-N_i$ if $k_i>N_i$. Note that $L_m(3)=l$ by Lemma \ref{Overall_Length}. We first set out to prove that $P_m(3)|P$. This can be done by proving $D^{l+P}(0,0,1)=D^l(0,0,1)$ or that $b_{l+P} \equiv b_l \; \text{mod} \; m$. Since $6|P$ and $p_i^{\alpha_i}x_i|P$ for $1 \leq i \leq r$, we have 
\[b_{l+P} \equiv b_l \; \text{mod} \; 2^l\]
\[b_{l+P} \equiv b_l \; \text{mod} \; p_i^{k_i}\]
which gives 
\[b_{l+P}-b_l \equiv 0 \; \text{mod} \; 2^l\]
\[b_{l+P}-b_l \equiv 0 \; \text{mod} \; p_i^{k_i}\]
and \[b_{l+P} -b_l \equiv 0 \; \text{mod} \; m\]
\[b_{l+p} \equiv b_l \; \text{mod} \; m\]
Since $l+P, l \equiv t \; \text{mod} \; 6$ for some $t, 0 \leq t \leq 6$, it will follow that $a_{l+P} \equiv a_l  \; \text{mod} \; m$ and $c_{l+P} \equiv c_l \; \text{mod} \; m$ by Lemma \ref{coefficients_mod_6}. Therefore, $D^{l+P}(0,0,1)=D^l(0,0,1)$ and $P_m(3)|P$.

\indent By Lemma \ref{divisor_divides_period}, $p_i^{\alpha_i}x_i|P_m(3)$ for every $i$. Therefore, $P|P_m(3)$ and $P_m(3)=P$ follows.
 
\end{proof}
  
\end{document}